\documentclass[11pt]{article}
\usepackage[T1]{fontenc}
\usepackage[utf8]{inputenc}
\usepackage[margin=2.54cm]{geometry}
\usepackage{microtype}
\usepackage{amsmath,amssymb,amsthm}
\usepackage{aliascnt}
\numberwithin{equation}{section}
\allowdisplaybreaks
\usepackage{graphicx}
\usepackage{float}
\usepackage{enumitem}
\usepackage{caption}
\usepackage{xcolor}
\usepackage{hyperref}
\definecolor{mycitegreen}{HTML}{00A99D}
\hypersetup{
  colorlinks=true,
  linkcolor={blue!85!black},
  citecolor={red!85!black},
  urlcolor=magenta
}
\usepackage[noabbrev]{cleveref}
\usepackage{tikz}
\usetikzlibrary{arrows.meta,decorations.markings,positioning,calc,shapes.geometric,fit,backgrounds,decorations.pathmorphing}

\newcommand{\manualstealtharc}[5]{%
    \coordinate (#1-arcstart)
        at ([shift={(#3:#2cm)}]#1.center);
    \coordinate (#1-arcend)
        at ([shift={(#4:#2cm)}]#1.center);

    \draw[line width=1.05pt]
        (#1-arcstart)
        arc[
            start angle=#3,
            end angle={#4-12},
            radius=#2cm
        ];

    \pgfmathsetmacro{\arrowangle}{#4+90+#5}
    \begin{scope}[
        shift={(#1-arcend)},
        rotate=\arrowangle
    ]
        \fill
            (0,0)
            -- (-0.25,0.095)
            -- (-0.18,0)
            -- (-0.25,-0.095)
            -- cycle;
    \end{scope}
}

\theoremstyle{plain}
\newtheorem{theorem}{Theorem}[section]

\newaliascnt{lemma}{theorem}
\newtheorem{lemma}[lemma]{Lemma}
\aliascntresetthe{lemma}
\newaliascnt{proposition}{theorem}
\newtheorem{proposition}[proposition]{Proposition}
\aliascntresetthe{proposition}
\newaliascnt{corollary}{theorem}
\newtheorem{corollary}[corollary]{Corollary}
\aliascntresetthe{corollary}
\newaliascnt{problem}{theorem}
\newtheorem{problem}[problem]{Problem}
\aliascntresetthe{problem}
\newaliascnt{claim}{theorem}
\newtheorem{claim}[claim]{Claim}
\aliascntresetthe{claim}
\theoremstyle{definition}
\newaliascnt{definition}{theorem}
\newtheorem{definition}[definition]{Definition}
\aliascntresetthe{definition}
\theoremstyle{remark}
\newaliascnt{remark}{theorem}
\newtheorem{remark}[remark]{Remark}
\aliascntresetthe{remark}
\crefname{theorem}{Theorem}{Theorems}
\crefname{lemma}{Lemma}{Lemmas}
\crefname{proposition}{Proposition}{Propositions}
\crefname{corollary}{Corollary}{Corollaries}
\crefname{problem}{Problem}{Problems}
\crefname{claim}{Claim}{Claims}
\crefname{definition}{Definition}{Definitions}
\crefname{remark}{Remark}{Remarks}
\crefname{section}{Section}{Sections}
\crefname{figure}{Figure}{Figures}
\crefname{subsection}{Subsection}{Subsections}
\newcommand{\eps}{\varepsilon}
\newcommand{\rn}{\operatorname{RN}}
\newcommand{\sigmapm}{\sigma^{+-}}
\newcommand{\dzero}{\delta^0}
\newcommand{\calL}{\mathcal L}
\newcommand{\calK}{\mathcal K}
\tikzset{
  vertex/.style={circle,draw=black,fill=white,line width=.8pt,
    inner sep=0pt,minimum size=4.0pt},
  branch/.style={circle,draw=black,fill=white,line width=.8pt,
    inner sep=0pt,minimum size=4.0pt},
  block/.style={rectangle,rounded corners=2.5pt,draw=black,fill=white,
    line width=.8pt,inner sep=4pt},
  smallblock/.style={rectangle,rounded corners=2.5pt,draw=black,fill=white,
    line width=.8pt,inner sep=4pt},
  module/.style={rounded corners=4pt,draw=black,fill=white,line width=.8pt},
  part/.style={circle,draw=black,fill=white,minimum size=3.0cm,
    align=center,line width=.9pt,inner sep=2pt},
  demand/.style={-{Stealth[length=2.7mm,width=1.9mm]},line width=1pt,draw=black},
  weakdemand/.style={-{Stealth[length=2.7mm,width=1.9mm]},line width=1pt,dashed,draw=black},
  pathline/.style={-{Stealth[length=2.6mm,width=1.8mm]},line width=1pt,draw=black},
  switch/.style={-{Stealth[length=2.6mm,width=1.8mm]},dashed,line width=1pt,draw=black},
  rung/.style={-{Stealth[length=2.6mm,width=1.8mm]},line width=1pt,draw=black},
  connector/.style={-{Stealth[length=2.6mm,width=1.8mm]},line width=1pt,draw=black}
}
\title{Spanning Subdivisions in Robustly Expanding Digraphs}
\author{
Zhilan Wang\thanks{School of Mathematics, Shandong University, Jinan 250100, China. Supported by the National Natural Science Foundation of China (No.~12571373).} \thanks{Corresponding author. Email: \href{mailto:zhilanwang@mail.sdu.edu.cn}{\texttt{zhilanwang@mail.sdu.edu.cn}}}
\and
Shuo Wei\footnotemark[1]
\and
Jin Yan\footnotemark[1] 
}
\date{}
\begin{document}
\maketitle
\vspace{-1em}
\begin{abstract}

\noindent K\"uhn, Osthus and Treglown [\emph{J. Combin. Theory Ser. B} \textbf{100} (2010)] proved that every sufficiently large robust outexpander with linear minimum semi-degree contains a Hamilton cycle. In this paper, we extend their result to spanning subdivisions of a digraph $F$ in a robust outexpander $D$, with prescribed branch vertices and arbitrarily specified  path lengths. In particular, the digraph $F$ need not be fixed. For an $n$-vertex digraph $D$, $F$ may have as many as a constant multiple of $n/\log n$ arcs. This holds provided that $F$ has sufficiently few vertices and the prescribed lengths satisfy suitable conditions. We further establish a more general result for mixed prescribed path lengths, which allows $F$ to have linearly many arcs when the shorter replacement paths together use only a small proportion of the vertices of $D$.
As applications, we obtain spanning subdivision extensions of the asymptotic Nash--Williams theorem for digraphs and an asymptotically sharp Ore-type theorem for oriented graphs.
\end{abstract}



\section{Introduction}

Every sufficiently large robust outexpander with linear minimum semi-degree contains a Hamilton cycle~\cite{KOT10,LoPatel2018}. Robust expansion also underlies Hamilton decomposition results~\cite{KO13}.

An $F$-subdivision is obtained by replacing each arc of a digraph $F$ with a directed path, with all internal vertices distinct and disjoint from the branch vertices. It is \emph{spanning} if its vertex set is $V(D)$.  We ask whether the robust expansion  and semi-degree hypotheses also allow the branch map and the length  of every replacement path to be prescribed simultaneously, even when $F$ grows with the host digraph $D$.

Prescribed branch vertices arise naturally in directed linkage and ordered Hamiltonicity problems~\cite{Ferrara,KO08,KOY08}. Cheng, Wang and Yan~\cite{Linked} proved a sharp Dirac-type theorem
for prescribed spanning subdivisions of every fixed digraph without isolated vertices, with both the branch map and all subdivision  path lengths prescribed. Their result also allows $F$ to grow, subject to  a bound on the total length of short subdivision paths. For oriented graphs, Zhou and Yan~\cite{ZhouYan2025} established a minimum semi-degree
condition for arbitrary $H$-linkage with prescribed path lengths. Related spanning subdivision problems with unspecified branch vertices were studied by Wang, Cheng and Yan~\cite{Spanning}, and the undirected problem with prescribed path lengths was considered by Wang, Wei and Yan~\cite{WW26}.

\subsection{Main result}
All path lengths are measured in arcs. For an injection $f:V(F)\to V(D)$ and a function $\ell:A(F)\to\mathbb N$, a \emph{spanning $(F,f,\ell)$-subdivision} in $D$ is a spanning $F$-subdivision with branch map $f$ in which
the path replacing each arc $e\in A(F)$ has length $\ell(e)$. For an $n$-vertex host digraph, we call $\ell$
\emph{admissible} if $|V(F)|+\sum_{e\in A(F)}(\ell(e)-1)=n$, since the path replacing $e$ contributes exactly $\ell(e)-1$
internal vertices.

For an $n$-vertex digraph $D$, the $\nu$-robust outneighbourhood of a set $S\subseteq V(D)$ is
\[
 \rn^+_{\nu,D}(S):=\{x\in V(D):|N_D^-(x)\cap S|\ge\nu n\}.
\]
We say that $D$ is a \emph{robust $(\nu,\tau)$-outexpander} if $|\rn^+_{\nu,D}(S)|\ge|S|+\nu n$ whenever $\tau n\le|S|\le(1-\tau)n$.  A \emph{robust $(\nu,\tau)$-expander}  is both a robust $(\nu,\tau)$-outexpander and a robust $(\nu,\tau)$-inexpander. Write $\dzero(D):=\min\{\delta^+(D),\delta^-(D)\}$ for its minimum semi-degree.

For $\alpha,\mu>0$, $C\in\mathbb N$ and $n\ge2$, let $\mathcal{S}_n(\alpha,\mu,C)$ denote the class of all pairs
$(F,\ell)$ such that $F$ is a digraph with $|V(F)|\le\mu n$ and $|A(F)|\le\alpha n/\log n$, and $\ell:A(F)\to\mathbb N$ is admissible with $\ell(e)\ge C$ for every $e\in A(F)$.

\begin{theorem}\label{thm3}
Let $0<\nu\le\tau\ll\gamma<1$. Then there exist constants $\alpha,\mu>0$ and $C,n_0\in\mathbb N$ such that the following
holds. Every $n$-vertex robust $(\nu,\tau)$-outexpander $D$ with $n\ge n_0$ and $\dzero(D)\ge\gamma n$ contains a spanning
$(F,f,\ell)$-subdivision for every $(F,\ell)\in\mathcal{S}_n(\alpha,\mu,C)$ and every injection $f:V(F)\to V(D)$.
\end{theorem}

\cref{thm3} allows the target digraph $F$ to grow with $n$, with both $C$ and $n_0$ independent of $F$. More generally, \cref{thm:mixed-length} replaces the bound on $|A(F)|$ by the condition $\sum_{e:\,\ell(e)-1<T\log n}(\ell(e)-1)\le\beta n$
for suitable constants $T,\beta>0$, while retaining the bound on $|V(F)|$ and a constant lower bound on the prescribed path lengths. In particular, this permits $|A(F)|=\Theta(n)$ for certain admissible length assignments.

\subsection{Applications to degree conditions}

The study of degree conditions for directed Hamilton cycles goes back to Ghouila-Houri~\cite{GH} and Woodall~\cite{Woodall}. Nash--Williams~\cite{NW75} conjectured a directed degree-sequence analogue of Chv\'atal's theorem~\cite{Chvatal1972}. An asymptotic form of this conjecture was proved by K\"uhn, Osthus and
Treglown~\cite{KOT10}. Their degree-sequence condition, recalled in \cref{subsec:robust-definitions}, implies robust outexpansion and linear minimum semi-degree. Hence \cref{thm3} yields the following corollary.

\begin{corollary}\label{thm2}
For every $0<\gamma<1/2$, there exist constants $\alpha,\mu>0$ and $C,n_0\in\mathbb N$ such that the following holds.
Every $n$-vertex digraph $D$ with $n\ge n_0$ satisfying the asymptotic Nash--Williams condition with parameter $\gamma$
contains a spanning $(F,f,\ell)$-subdivision for every $(F,\ell)\in\mathcal{S}_n(\alpha,\mu,C)$ and every injection
$f:V(F)\to V(D)$.
\end{corollary}

Our second application concerns oriented graphs. Write $\sigmapm(D):=\min\{d^+(x)+d^-(y): x,y\in V(D),\ x\ne y,\ xy\notin A(D)\}$. Kelly, K\"uhn and Osthus~\cite{KKO08} proved that $\sigmapm(D)\ge(3/4+\eps)n$ guarantees a Hamilton cycle for
every fixed $\eps>0$ and sufficiently large $n$. Chang, Cheng, Dai, Ouyang and Wang~\cite{CCD} proved that $\sigmapm(D)\ge(3n-3)/4$ suffices for Hamiltonicity.

We show that a fixed linear excess above $3n/4$ also forces robust outexpansion. By \cref{KKO08}, this condition also implies a  linear minimum semi-degree bound, \cref{thm3} gives the following strengthening.

\begin{theorem}\label{thm4}
For every $\eps>0$, there exist constants $\alpha,\mu>0$ and $C,n_0\in\mathbb N$ such that the following holds. Every oriented graph $D$ on $n\ge n_0$ vertices with $\sigmapm(D)\ge(3/4+\eps)n$ contains a spanning $(F,f,\ell)$-subdivision for every $(F,\ell)\in\mathcal{S}_n(\alpha,\mu,C)$ and every injection $f:V(F)\to V(D)$.
\end{theorem}

The construction in \cref{prop:sharp-Ore} gives oriented graphs with $\sigmapm(D)\ge3n/4+(3q-4)/4$ that do not admit a spanning linkage joining $q$ prescribed ordered pairs. Thus the degree condition in \cref{thm4} is asymptotically sharp.
\medskip

\noindent\textbf{Structure of the paper.}
\cref{sec:prelim} collects the notation and required tools. In \cref{sec:main-proof}, we prove \cref{thm1} by absorbing paths cut from a leftover Hamilton cycle into their assigned subdivision paths. In \cref{sec:further}, we first establish a lemma for paths of prescribed length and use it to prove the mixed-length result
\cref{thm:mixed-length}. We then deduce \cref{thm3} and the Nash--Williams application. \cref{sec:ore} treats the Ore-type condition, with supplementary estimates in \cref{app:ore-estimates}. We conclude with further questions in \cref{sec:concluding}.

\section{Preliminaries}\label{sec:prelim}

\subsection{Notation}\label{subsec:subdivisions}
All digraphs are finite and loopless, with at most one arc in each direction between two distinct vertices. An oriented graph is a digraph with at most one arc between any two distinct vertices. Let $D=(V,A)$ be a digraph. For $x\in V(D)$, let $N_D^+(x)$ and $N_D^-(x)$ denote its out- and inneighbourhoods, and write
$d_D^+(x):=|N_D^+(x)|$ and $d_D^-(x):=|N_D^-(x)|$. For $X\subseteq V(D)$ and $\sigma\in\{+,-\}$, set
$N_D^\sigma(x,X):=N_D^\sigma(x)\cap X$ and $d_D^\sigma(x,X):=|N_D^\sigma(x,X)|$. We write $\delta^+(D)$ and $\delta^-(D)$ for the minimum outdegree and minimum indegree of $D$, respectively, and define its minimum
semi-degree by $\dzero(D):=\min\{\delta^+(D),\delta^-(D)\}$.

For $X,Y\subseteq V(D)$, let $e_D(X,Y)$ denote the number of arcs directed from $X$ to $Y$; the sets need not be disjoint. We write $D-X$ for the digraph obtained by deleting $X$, and $D[X]$ for the subdigraph induced by $X$. All paths and cycles are directed unless stated otherwise. The length of a path is its number of arcs,
and its order is its number of vertices. For a set $U$, let $U^{[2]}:=\{(x,y)\in U^2:x\ne y\}$. All logarithms are natural.

Let $F$ be a digraph. An $F$-subdivision in $D$ consists of an injection $f:V(F)\to V(D)$ and, for each arc $e=ab\in A(F)$, a directed path $P_e$ from $f(a)$ to $f(b)$ such that these paths are internally vertex-disjoint and their internal vertices avoid $f(V(F))$. Its vertex set is $f(V(F))\cup\bigcup_{e\in A(F)}V(P_e)$.

We use the standard hierarchy notation. For example, $0<a\ll b\ll c<1$ means that the constants are chosen from
right to left, each sufficiently small in terms of those to its right.

\subsection{Robust expansion and degree-sequence conditions}\label{subsec:robust-definitions}

We shall repeatedly use the following standard robustness facts.

\begin{proposition}[\cite{LoPatel2018}, Proposition~7]
\label{prop:delete}
Let $0<\nu\le\tau<1/2$, and let $D$ be an $n$-vertex robust $(\nu,\tau)$-outexpander. If $W\subseteq V(D)$ and $|W|\le \nu n/4$, then $D-W$ is a robust $(\nu/2,2\tau)$-outexpander. The analogous assertion holds for robust inexpansion and hence for robust expansion.
\end{proposition}
\begin{proposition}[\cite{LoPatel2018}, Proposition~8]\label{prop:out-to-in}
Let $D$ be an $n$-vertex robust $(\nu,\tau)$-outexpander with $\dzero(D)\ge\gamma n$. If $\gamma>2\tau$, $\tau\gamma>\nu^2/2$, and $\nu<1/2$, then $D$ is a robust $(\nu^2/2,2\tau)$-inexpander. In particular, after weakening the parameters, a robust outexpander with linear minimum semi-degree may be treated as a robust expander.
\end{proposition}
\begin{lemma}[\cite{LoPatel2018}, Lemma~10]\label{lem:short-paths}
Let $0<\nu\le\tau\le\gamma/4<1/4$. Suppose that $D$ is an $n$-vertex robust $(\nu,\tau)$-expander with $\dzero(D)\ge \gamma n$. If $u_1,\ldots,u_r,v_1,\ldots,v_r$ are $2r$ distinct vertices and $n\ge (6r+11)\nu^{-2}$, then $D$ contains pairwise vertex-disjoint paths $P_1,\ldots,P_r$ such that $P_i$ is directed from $u_i$ to $v_i$ and each $P_i$ has order at most $2\nu^{-1}+1$.
\end{lemma}
Let $d_1^+\le\cdots\le d_n^+$ and $d_1^-\le\cdots\le d_n^-$ denote the outdegree and indegree sequences of an $n$-vertex digraph $D$, respectively. For $0<\gamma<1/2$, say that an $n$-vertex digraph $D$ satisfies the \emph{asymptotic Nash--Williams condition with parameter $\gamma$} if, for every integer $i$ with $1\le i<n/2$, either $d_i^+\ge i+\gamma n$ or $d^-_{n-i-\gamma n}\ge n-i$, and either $d_i^-\ge i+\gamma n$ or $d^+_{n-i-\gamma n}\ge n-i$, where integer parts in the subscripts are suppressed.

The proof of the asymptotic Nash--Williams theorem by K\"uhn, Osthus and Treglown~\cite{KOT10} gives the following robust expansion consequence.
\begin{lemma}[\cite{KOT10}, Lemma~13]\label{lem:NW-robust}
For every $0<\gamma<1/2$ there exist constants $\gamma_0>0$, $0<\nu\le\tau\ll\gamma_0$ and $n_0\in\mathbb N$ such that every $n$-vertex digraph $D$ with $n\ge n_0$ satisfying the asymptotic Nash--Williams condition with parameter $\gamma$ has $\dzero(D)\ge \gamma_0 n$ and is a robust $(\nu,\tau)$-outexpander.
\end{lemma}

\begin{lemma}[\cite{LoPatel2018}, Theorem~3]\label{lem:hamilton-robust}
Let $0<\nu\le\tau\le\gamma/16<1/16$, and let $n$ be sufficiently large. If $D$ is an $n$-vertex robust $(\nu,\tau)$-outexpander with $\dzero(D)\ge \gamma n$, then $D$ contains a Hamilton cycle.
\end{lemma}
The next lemma gives short paths avoiding a small forbidden set.
\begin{lemma}[Short paths avoiding a small set]
\label{lem:short-avoid}
Let $0<\beta\ll \nu\le\tau\ll\gamma<1$, let $D$ be an $n$-vertex robust $(\nu,\tau)$-expander with $\dzero(D)\ge \gamma n$, and let $W\subseteq V(D)$ satisfy $|W|\le \beta n$. Suppose that $u_1,\ldots,u_r,v_1,\ldots,v_r$ are distinct vertices of $V(D)\setminus W$ and that $r\le \beta n$. Then there are pairwise vertex-disjoint paths $P_1,\ldots,P_r$ in $D-W$ such that $P_i$ is directed from $u_i$ to $v_i$ and each $P_i$ has order bounded above by a constant depending only on $\nu$.
\end{lemma}
\begin{proof}
Set $D':=D-W$ and $n':=|V(D')|$. By \cref{prop:delete}, $D'$ is a robust $(\nu',\tau')$-expander, where $\nu':=\nu/2$ and $\tau':=2\tau$. Moreover, $\dzero(D')\ge \gamma n-|W|\ge (\gamma-\beta)n\ge \gamma n'/2$. Since $r\le\beta n$ and $\beta\ll\nu^2$, we have $n'\ge(6r+11)(\nu')^{-2}$ for all sufficiently large $n$. Thus \cref{lem:short-paths} applies to $D'$ and gives the required paths.
\end{proof}

\subsection{Ladders}
We use the ladder absorber introduced by Lo and Patel~\cite{LoPatel2018}. Let
$\{v_0,v_1,\ldots,v_{2k},v^*_{2k},v^*_{2k-1},\ldots,v^*_0\}\subseteq V(D)$.
An alternating path from $v_0$ to $v^*_0$ is the subdigraph with this vertex set and arc set
$\{v_i v_{i+1},\ v^*_{i+1}v^*_i: i=0,2,\ldots,2k-2\}\cup\{v_{2k}v^*_{2k}\}\cup\{v_{j+1}v_j,\ v^*_jv^*_{j+1}: j=1,3,\ldots,2k-1\}$.
\begin{definition}[\cite{LoPatel2018}, Definition 14]\label{def:ladder}
A \emph{ladder} $L$ from $v_0$ to $v^*_0$ is a digraph
$L=Q\cup Q_1\cup Q_3\cup\cdots\cup Q_{2k-1}$,
where $Q$ is the alternating path above and each $Q_i$ is a directed path from $v_i$ to $v^*_i$, internally disjoint from $Q$ and from all other $Q_j$.  For $i=0,2,\ldots,2k-2$, define the rung path
$R_i:=v_i v_{i+1} Q_{i+1} v^*_{i+1}v^*_i$,
and set $R_{2k}:=v_{2k}v^*_{2k}$.  For $i=2,4,\ldots,2k$, define the alternative rung path
$R'_i:=v_i v_{i-1} Q_{i-1} v^*_{i-1}v^*_i$.
When a directed path from $v_0$ to $v^*_0$ is supplied, we denote it by $R'_0$. We say that a ladder $L$ from $v_0$ to $v^*_0$ \emph{is embedded in} an $F$-subdivision $F_0$ if all its rung paths $R_0,R_2,\ldots,R_{2k}$ occur, in this order, as subpaths of one subdivided path of $F_0$.
\end{definition}
The two kinds of rung paths used in the switching argument are depicted in \cref{fig:ladder}.
\begin{figure}[t]
\centering
\begin{tikzpicture}[x=1cm,y=1cm,every node/.style={font=\small}]

  \node at (2.55,2.20) {\emph{before switching}};
  \node at (10.05,2.20) {\emph{after switching}};
  \node at (6.30,-0.85) {$\Longrightarrow$};

  \node[vertex,label=below:{$v_0$}] (v0) at (0,1.10) {};
  \node[vertex,label=below:{$v_1$}] (v1) at (1.50,1.10) {};
  \node[vertex,label=below:{$v_1^*$}] (v1s) at (3.60,1.10) {};
  \node[vertex,label=below:{$v_0^*$}] (v0s) at (5.10,1.10) {};
  \draw[connector] (v0)--(v1);
  \draw[rung] (v1) to[bend left=25] node[above] {$R_0$} (v1s);
  \draw[connector] (v1s)--(v0s);

  \node[vertex,label=below:{$v_2$}] (v2) at (0,-0.85) {};
  \node[vertex,label=below:{$v_3$}] (v3) at (1.50,-0.85) {};
  \node[vertex,label=below:{$v_3^*$}] (v3s) at (3.60,-0.85) {};
  \node[vertex,label=below:{$v_2^*$}] (v2s) at (5.10,-0.85) {};
  \draw[connector] (v2)--(v3);
  \draw[rung] (v3) to[bend left=25] node[above] {$R_2$} (v3s);
  \draw[connector] (v3s)--(v2s);

  \node[vertex,label=below:{$v_4$}] (v4) at (0,-2.80) {};
  \node[vertex,label=below:{$v_5$}] (v5) at (1.50,-2.80) {};
  \node[vertex,label=below:{$v_5^*$}] (v5s) at (3.60,-2.80) {};
  \node[vertex,label=below:{$v_4^*$}] (v4s) at (5.10,-2.80) {};
  \draw[connector] (v4)--(v5);
  \draw[rung] (v5) to[bend left=25] node[above] {$R_4$} (v5s);
  \draw[connector] (v5s)--(v4s);

  \node[vertex,label=below:{$v_0$}] (w0) at (7.50,1.10) {};
  \node[vertex,label=below:{$v_0^*$}] (w0s) at (12.60,1.10) {};
  \draw[weakdemand] (w0) to[bend left=12] node[above] {$P$} (w0s);

  \node[vertex,label=below:{$v_2$}] (w2) at (7.50,-0.85) {};
  \node[vertex,label=below:{$v_1$}] (w1) at (9.00,-0.85) {};
  \node[vertex,label=below:{$v_1^*$}] (w1s) at (11.10,-0.85) {};
  \node[vertex,label=below:{$v_2^*$}] (w2s) at (12.60,-0.85) {};
  \draw[connector] (w2)--(w1);
  \draw[rung] (w1) to[bend left=25] node[above] {$R'_2$} (w1s);
  \draw[connector] (w1s)--(w2s);

  \node[vertex,label=below:{$v_4$}] (w4) at (7.50,-2.80) {};
  \node[vertex,label=below:{$v_3$}] (w3) at (9.00,-2.80) {};
  \node[vertex,label=below:{$v_3^*$}] (w3s) at (11.10,-2.80) {};
  \node[vertex,label=below:{$v_4^*$}] (w4s) at (12.60,-2.80) {};
  \draw[connector] (w4)--(w3);
  \draw[rung] (w3) to[bend left=25] node[above] {$R'_4$} (w3s);
  \draw[connector] (w3s)--(w4s);

  \node at (2.55,-3.75) {$\vdots$};
  \node at (10.05,-3.75) {$\vdots$};

\end{tikzpicture}
\caption{Ladder switching: the rung paths $R_0,R_2,\ldots,R_{2k}$ are replaced by $P,R'_2,\ldots,R'_{2k}$. Only the first three rung paths and their replacements are shown ($k\ge3$).}
\label{fig:ladder}
\end{figure}
\begin{lemma}
\label{lem:ladder-absorb}
Let $F_0$ be an $F$-subdivision in a digraph $D$, and suppose that a ladder $L$ from $u$ to $v$ is embedded in one subdivided path of $F_0$.  Let $P$ be a directed path from $u$ to $v$ whose internal vertices are disjoint from $V(F_0)$.  Then there is an $F$-subdivision $F'_0$ with the same branch map as $F_0$ and with $V(F'_0)=V(F_0)\cup V(P)$. Moreover, only the subdivided path containing $L$ is changed.
\end{lemma}
\begin{proof}
Write the rung paths of $L$ as $R_0,R_2,\ldots,R_{2k}$ and the alternative rung paths as $R'_2,R'_4,\ldots,R'_{2k}$, and set $R'_0:=P$. Starting with $F_0$, replace $R_0$ by $R'_0$, then $R_2$ by $R'_2$, and continue in this order. At each step, the old and new subpaths have the same endpoints. The first replacement is valid because the internal vertices of $P$ lie outside $V(F_0)$. At each subsequent step, the internal vertices of the alternative rung path have been released by the preceding switch, so they avoid the remainder of the current subdivision. Thus every switch preserves an $F$-subdivision with the same branch map.
For each $j=1,\ldots,k$, the path $Q_{2j-1}$ is removed with $R_{2j-2}$ and restored by $R'_{2j}$, while $R_{2k}$ has no internal vertices. Hence the final vertex set is $V(F_0)\cup V(P)$, and only the subdivided path containing $L$ has changed.
\end{proof}

\begin{lemma}[\cite{LoPatel2018}, Lemma 21]\label{lem:disjoint-ladders}
Let $0<\nu\le\tau\le\gamma/16<1/16$.  Suppose that $D$ is an $n$-vertex robust $(\nu,\tau)$-expander with $\dzero(D)\ge \gamma n$.  If $u_1,\ldots,u_l,v_1,\ldots,v_l$ are $2l$ distinct vertices and $n\ge 460l\nu^{-3}$, then $D$ contains pairwise vertex-disjoint ladders $L_{u_1,v_1},\ldots,L_{u_l,v_l}$ such that each ladder has order at most $12\nu^{-2}$ and its alternating path has order at most $8\nu^{-1}$.
\end{lemma}
\begin{lemma}[Ladders avoiding a small set]
\label{lem:ladder-avoid}
Let $0<\beta\ll \nu\le\tau\ll\gamma<1$, let $D$ be an $n$-vertex robust $(\nu,\tau)$-expander with $\dzero(D)\ge \gamma n$, and let $W\subseteq V(D)$ satisfy $|W|\le \beta n$. Suppose that $u_1,\ldots,u_l,v_1,\ldots,v_l$ are distinct vertices of $V(D)\setminus W$ and that $l\le\beta n$. Then $D-W$ contains pairwise vertex-disjoint ladders from $u_i$ to $v_i$, each of order bounded above by a constant depending only on $\nu$.
\end{lemma}
\begin{proof}
As in the proof of \cref{lem:short-avoid}, $D-W$ is a robust $(\nu/2,2\tau)$-expander with $\dzero(D-W)\ge(\gamma/2)|V(D-W)|$. Since $l\le\beta n$ and $\beta\ll\nu^3$, we have $|V(D-W)|\ge460l(\nu/2)^{-3}$, so \cref{lem:disjoint-ladders} applies in $D-W$.
\end{proof}

\subsection{Covering pairs}
For distinct vertices $x,y,u,v$ in a digraph $D$, the ordered pair
$(u,v)$ \emph{covers} $(x,y)$ if $ux,yv\in A(D)$.
A family $\calK\subseteq V(D)^{[2]}$ \emph{$d$-covers} a set
$U\subseteq V(D)$ if every $(x,y)\in U^{[2]}$ is covered by at least
$d$ members of $\calK$.
\begin{lemma}[\cite{LoPatel2018}, Lemma 19]\label{lem:covering-pairs}
For every $\gamma>0$ and integer $d\ge8$ there are constants $c=c(d,\gamma)$ and $n_0=n_0(d,\gamma)$ such that the following holds.  Let $D$ be an $n$-vertex digraph with $n\ge n_0$ and $\dzero(D)\ge \gamma n$.  For every $U\subseteq V(D)$ there exists a family $\calK\subseteq V(D)^{[2]}$ of pairwise vertex-disjoint ordered pairs
such that $\calK$ $d$-covers $U$ and $|\calK|\le c\log n$.
\end{lemma}
\begin{lemma}[Covering pairs avoiding a small set]
\label{lem:covering-avoid}
Let $\gamma>0$ and let $d\ge8$.  There exist constants $c=c(d,\gamma)>0$, $\beta_0=\beta_0(d,\gamma)>0$ and $n_0\in\mathbb N$ such that the following holds.  Let $D$ be an $n$-vertex digraph with $n\ge n_0$ and $\dzero(D)\ge \gamma n$.  Let $W\subseteq V(D)$ satisfy $|W|\le\beta_0 n$, and set $U:=V(D)\setminus W$.  Then there exists a family $\calK\subseteq U^{[2]}$ of pairwise vertex-disjoint ordered pairs such that $\calK$ $d$-covers $U$ in $D$ and $|\calK|\le c\log n$.
\end{lemma}
\begin{proof}
Apply \cref{lem:covering-pairs} to the induced digraph $D[U]$ with semi-degree parameter $\gamma/2$.  Indeed, if $\beta_0\ll\gamma$, then
$\dzero(D[U])\ge \gamma n-|W|\ge \gamma n/2\ge (\gamma/2)|U|$.
The resulting family $\calK\subseteq U^{[2]}$ has size at most $c(d,\gamma/2)\log|U|\le c\log n$ and $d$-covers $U$ in $D[U]$, hence also in $D$.
\end{proof}

\section{Proof of the logarithmic length theorem}\label{sec:main-proof}

We first prove the following logarithmic-length version of \cref{thm3}. It provides the absorption step used later to obtain the constant lower bound on path lengths.

\begin{theorem}\label{thm1}
Let $0<1/n_0\ll \alpha\ll \mu\ll 1/C\ll \nu\le \tau\ll \gamma<1$.
Suppose that $D$ is an $n$-vertex robust $(\nu,\tau)$-outexpander with $n\ge n_0$ and $\dzero(D)\ge \gamma n$.
Then, for every
$(F,\ell)\in\mathcal{S}_n(\alpha,\mu,\lceil C\log n\rceil)$
and every injection $f:V(F)\to V(D)$,
the digraph $D$ contains a spanning $(F,f,\ell)$-subdivision.
\end{theorem}

\subsection{Arc-distributed absorbers}\label{subsec:absorbers}
Let $F$ be a digraph with $A(F)=\{e_1,\ldots,e_q\}$, where $e_i=a_i b_i$, and let $f:V(F)\to V(D)$ be an injection. To control the subdivision path lengths separately, we assign a private covering family and a collection of ladders to each arc. We embed these ladders in a small $F$-subdivision. A Hamilton cycle on the remaining vertices will then be partitioned into paths of suitable orders, each of which is absorbed into its assigned subdivision path. The absorber is illustrated in \cref{fig:distributed}.

\begin{figure}[H]
\centering
\begin{tikzpicture}[
  x=1cm,y=1cm,
  every node/.style={font=\normalfont\normalsize}
]
  \tikzset{
    smallbranch/.style={
      circle,
      draw=black,
      fill=white,
      line width=.8pt,
      inner sep=0pt,
      minimum size=4pt
    },
    smallbox/.style={
      draw=black,
      rounded corners=4pt,
      fill=white,
      line width=.8pt,
      minimum height=.76cm
    }
  }

  \node[smallbranch] (a) at (0,0) {};
  \node[left=0.25cm of a] {$f(a_i)$};

  \node[smallbranch] (b) at (10,0) {};
  \node[right=0.25cm of b] {$f(b_i)$};

  \draw[->,line width=.8pt] (a) -- (b);
  \node at (1.33,-0.6) {$P_i^0$};

  \node[
    draw=black,
    ellipse,
    fill=white,
    line width=.8pt,
    minimum width=4.2cm,
    minimum height=1.25cm
  ] (L) at (5,0) {private ladders};

  \node[
    smallbox,
    minimum width=1.8cm
  ] (K) at (5,1.8) {$\mathcal K_i$};

  \draw[->,line width=.8pt]
    (K.south) -- (L.north);

  \node[
    smallbox,
    minimum width=1.7cm
  ] (Q) at (5,-2.0) {$Q_i$};

  \draw[->,line width=.8pt,dashed]
    (Q.north) -- (L.south);

\end{tikzpicture}
\caption{The private absorbing module for $e_i=a_ib_i$. The family $\mathcal K_i$ indexes private ladders embedded in $P_i^0$. Later, $Q_i$ is absorbed into $P_i^0$ through a suitable ladder.}
\label{fig:distributed}
\end{figure}
\begin{definition}
\label{def:distributed-absorber}
Let $d\in\mathbb N$. An \emph{arc-distributed $d$-absorber} for $(F,f)$ in $D$ is a tuple
$\mathcal A=(\calK_1,\ldots,\calK_q,\calL,F_0)$
satisfying the following properties.
\begin{enumerate}[label=\textup{(\roman*)}]
  \item For each $i\in[q]$, $\calK_i\subseteq V(D)^{[2]}$ is a family of pairwise vertex-disjoint ordered pairs, and the families $\calK_1,\ldots,\calK_q$ are mutually vertex-disjoint.
  \item $F_0$ is an $F$-subdivision in $D$ with branch map $f$. For each $i\in[q]$, let $P^0_i$ be the subdivided path of $F_0$ corresponding to $e_i$.
  \item For every $(u,v)\in\calK_i$, there is a ladder $L_{u,v}\in\calL$ from $u$ to $v$, and $L_{u,v}$ is embedded in $P^0_i$.
  \item All ladders in $\calL$ are pairwise vertex-disjoint.
  \item For each $i\in[q]$, the family $\calK_i$ $d$-covers $V(D)\setminus V(F_0)$.
\end{enumerate}
We write $p_i:=|A(P^0_i)|$ for the length of the subdivided path corresponding to $e_i$.
\end{definition}

Condition~\textup{(v)} ensures that each covering family can accommodate the relevant endpoints, which are determined only after the Hamilton cycle on the remaining vertices has been chosen and partitioned.

To apply the connecting lemma when several arcs share a branch vertex, we first choose distinct attachment vertices for all arcs.
\begin{lemma}
\label{lem:attachment-vertices}
Let $0<\alpha\ll\mu\ll\beta\ll\gamma<1$ and let $D$ be an $n$-vertex digraph with $\dzero(D)\ge\gamma n$. Let $F$ be a digraph satisfying $|V(F)|\le\mu n$ and $q:=|A(F)|\le \alpha n/\log n$, and let $f:V(F)\to V(D)$ be an injection. Let $W\subseteq V(D)$ satisfy $|W|\le\beta n$ and $f(V(F))\subseteq W$. Then, for every arc $e_i=a_i b_i\in A(F)$, one can choose vertices $s_i,t_i\in V(D)\setminus W$ such that $f(a_i)s_i\in A(D)$, $t_i f(b_i)\in A(D)$,
and the $2q$ vertices $s_1,t_1,\ldots,s_q,t_q$ are distinct.
\end{lemma}

\begin{proof}
Choose $s_1,\ldots,s_q$ greedily, always outside $W$ and the previously chosen attachment vertices. Then choose $t_1,\ldots,t_q$ greedily, avoiding $W$, all $s_j$, and the previously chosen $t_j$. At each step fewer than $|W|+2q+1\le \beta n+2\alpha n/\log n+1<\gamma n$ vertices are forbidden. The bounds $d^+(f(a_i)),d^-(f(b_i))\ge\gamma n$ therefore guarantee every required choice.
\end{proof}

\begin{lemma}[Arc-distributed absorber]
\label{lem:distributed-absorber}
Fix an integer $d\ge8$, and let $0<1/n_0\ll\alpha\ll\mu\ll\beta\ll\nu\le\tau\ll\gamma<1$, with the constants chosen after $d$. Suppose that $D$ is an $n$-vertex robust $(\nu,\tau)$-expander with $\dzero(D)\ge\gamma n$, where $n\ge n_0$. Let $F$ be a digraph satisfying $|V(F)|\le\mu n$ and $q:=|A(F)|\le\alpha n/\log n$, and let $f:V(F)\to V(D)$ be an injection. Then there is a constant $C_{\rm A}=C_{\rm A}(d,\nu,\tau,\gamma)$ such that $D$ contains an arc-distributed $d$-absorber $\mathcal A=(\calK_1,\ldots,\calK_q,\calL,F_0)$ for $(F,f)$ satisfying $|V(F_0)|\le |V(F)|+C_{\rm A}q\log n$ and $p_i\le C_{\rm A}\log n$ for every $i\in[q]$.
\end{lemma}

\begin{proof}
Choose the constants so that $\beta\le\beta_0(d,\gamma)$ and $\beta\ll\nu^3$, where $\beta_0$ is given by \cref{lem:covering-avoid}. Enumerate $A(F)=\{e_1,\ldots,e_q\}$, where $e_i=a_i b_i$.

We first construct the covering families. Set $W_0:=f(V(F))$. For $i=1,\ldots,q$, apply \cref{lem:covering-avoid} with forbidden set $W_{i-1}$ and $U_i:=V(D)\setminus W_{i-1}$. This gives a family $\calK_i\subseteq U_i^{[2]}$ of pairwise vertex-disjoint ordered pairs such that $\calK_i$ $d$-covers $U_i$ and
$|\calK_i|\le c_1\log n$, where $c_1$ depends only on $d$ and $\gamma$. Define
$W_i:=W_{i-1}\cup\bigcup_{(u,v)\in\calK_i}\{u,v\}$. For every $i\le q$,
$|W_i|\le |V(F)|+2ic_1\log n \le \mu n+2c_1\alpha n\le\beta n$.
Thus the forbidden set remains small enough for every application of \cref{lem:covering-avoid}, and the families $\calK_1,\ldots,\calK_q$ are mutually vertex-disjoint.

The total number of pairs in these families is at most $\sum_{i=1}^q|\calK_i|\le c_1q\log n\le c_1\alpha n\le\beta n$. Apply \cref{lem:ladder-avoid} with forbidden set $f(V(F))$ and with all pairs in $\calK_1\cup\cdots\cup\calK_q$ as the prescribed endpoints. We obtain mutually vertex-disjoint ladders $L_{u,v}$, one for each $(u,v)\in\calK_i$, all avoiding $f(V(F))$, and each of order at most a constant $L_0=L_0(\nu,\tau,\gamma)$. Let $\calL$ be the collection of all these ladders.

For each $i$, list the rung paths of the ladders corresponding to $\calK_i$ as
 $R_{i,1},R_{i,2},\ldots,R_{i,r_i}$, first by ladder and then in their natural order within each ladder. Let $x_{i,j}$ and $y_{i,j}$ be the initial and terminal vertices of $R_{i,j}$. Since each ladder has bounded order,
 $r_i\le L_0|\calK_i|\le L_0c_1\log n$. Let $W_L$ be the union of $f(V(F))$ and all vertices of the ladders. Then
 $|W_L|\le |V(F)|+L_0\sum_{i=1}^q|\calK_i| \le \mu n+L_0c_1\alpha n\le\beta n$.
By \cref{lem:attachment-vertices}, we can choose attachment vertices $s_i,t_i\in V(D)\setminus W_L$ such that $f(a_i)s_i\in A(D)$ and $t_i f(b_i)\in A(D)$ for every $i\in[q]$, and the $2q$ attachment vertices are distinct.

We now connect the rung paths into the subdivision paths. Let $X$ be the set of all initial and terminal vertices of the paths $R_{i,j}$, and set $W^*:=W_L\setminus X$. Then $|W^*|\le\beta n$. For each $i\in[q]$, prescribe the connections $s_i\to x_{i,1}$, $y_{i,j}\to x_{i,j+1}$ for $1\le j<r_i$, and
$y_{i,r_i}\to t_i$. All endpoints in these connections are distinct,
and the total number of connections is $\sum_{i=1}^q(r_i+1)\le q+L_0c_1q\log n\le\beta n$. By \cref{lem:short-avoid}, $D-W^*$ contains pairwise vertex-disjoint directed paths realizing all these connections. The vertices of the ladders outside $X$ are forbidden, and every $x_{i,j}$ and $y_{i,j}$ is an endpoint of one of the connecting paths. Hence the connecting paths meet the ladders only at the prescribed endpoints of the connections.

For each $i\in[q]$, concatenate the connecting paths and rung paths, and add the arcs $f(a_i)s_i$ and $t_i f(b_i)$ at the two ends. This gives a directed path $P_i^0$ from $f(a_i)$ to $f(b_i)$. These paths are internally vertex-disjoint, and their internal vertices avoid the branch set. Their union, together with the branch vertices corresponding to isolated vertices of $F$, therefore forms an $F$-subdivision $F_0$ with branch map $f$.

Every ladder $L_{u,v}$ with $(u,v)\in\calK_i$ is embedded in $P_i^0$. Since $|\calK_i|\le c_1\log n$ and both the ladders and the connecting paths have bounded order, there is a constant $C_{\rm A}=C_{\rm A}(d,\nu,\tau,\gamma)$ such that $|V(F_0)|\le |V(F)|+C_{\rm A}q\log n$
and $p_i\le C_{\rm A}\log n$ for every $i\in[q]$.

It remains to check the covering condition. Since $W_{i-1}\subseteq V(F_0)$ for every $i$, we have $V(D)\setminus V(F_0)\subseteq U_i$. Therefore $\calK_i$ $d$-covers $V(D)\setminus V(F_0)$ for every $i$, so $\mathcal A=(\calK_1,\ldots,\calK_q,\calL,F_0)$ is the required absorber.
\end{proof}

\begin{lemma}
\label{lem:absorb-paths}
Let $D$ be a digraph, and let $\mathcal A=(\calK_1,\ldots,\calK_q,\calL,F_0)$ be an arc-distributed $d$-absorber for $(F,f)$ in $D$. Let $Q_1,\ldots,Q_q$ be pairwise vertex-disjoint directed paths in $D-V(F_0)$, each of order at least two. Then $D$ contains an $F$-subdivision $F^*$ with branch map $f$ such that
$V(F^*)=V(F_0)\cup\bigcup_{i=1}^q V(Q_i)$. Moreover, if $P^*_i$ is the subdivided path of $F^*$ corresponding to $e_i$, then for every $i\in[q]$,
$|A(P^*_i)|=p_i+|V(Q_i)|$.
\end{lemma}
\begin{proof}
For each $i\in[q]$, write $Q_i=x_i\cdots y_i$. Since $\calK_i$ $d$-covers $V(D)\setminus V(F_0)$, there exists a pair $(u_i,v_i)\in\calK_i$ such that $u_i x_i,y_i v_i\in A(D)$. The families $\calK_i$ are pairwise vertex-disjoint, so the chosen pairs are distinct. Thus $R_i^*:=u_i x_i\cdots y_i v_i$ is a directed path from $u_i$ to $v_i$ whose internal vertices are disjoint from $F_0$.

By the definition of the absorber, the ladder $L_{u_i,v_i}$ is embedded in the subdivided path $P^0_i$ corresponding to $e_i$. Applying \cref{lem:ladder-absorb} to this ladder and the path $R^*_i$ inserts the vertices of $Q_i$ into that subdivided path and preserves its endpoints and the branch map. The paths $Q_1,\ldots,Q_q$ are pairwise vertex-disjoint, as are the chosen ladders, and each chosen ladder lies on its assigned subdivision path. Hence these replacements can be performed successively for $i=1,\ldots,q$.

The resulting $F$-subdivision $F^*$ has the required vertex set and branch map. Exactly $|V(Q_i)|$ vertices are added to the subdivided path corresponding to $e_i$, while its endpoints remain unchanged. Its length is therefore $p_i+|V(Q_i)|$.
\end{proof}

\subsection{Completing the proof}
\begin{proof}[\textbf{Proof of \cref{thm1}}]
The order identity and $|V(F)|\le\mu n<n$ imply that $q:=|A(F)|\ge1$. Choose an auxiliary constant $\beta$ with $\mu\ll\beta\ll1/C\ll\nu$. By \cref{prop:out-to-in}, after replacing $\nu$ by a smaller constant and $\tau$ by a larger constant if necessary, $D$ is a robust $(\nu,\tau)$-expander. We keep the same notation for these weakened parameters.

Apply \cref{lem:distributed-absorber} with $d=8$. It gives an arc-distributed absorber $\mathcal A=(\calK_1,\ldots,\calK_q,\calL,F_0)$ for $(F,f)$ and a constant $C_{\rm A}$ such that $|V(F_0)|\le |V(F)|+C_{\rm A}q\log n$ and
$p_i\le C_{\rm A}\log n$ for every $i\in[q]$. Since $C_{\rm A}$ depends only on the constants to the right of $1/C$ in the hierarchy, we may choose $C$ sufficiently large that $C\ge2C_{\rm A}+2$.

Moreover, $|V(F_0)|\le(\mu+C_{\rm A}\alpha)n\le \nu n/4$. Set $D':=D-V(F_0)$ and $n':=|V(D')|$. By \cref{prop:delete}, $D'$ is a robust $(\nu/2,2\tau)$-outexpander, and $\dzero(D')\ge\gamma n-|V(F_0)|\ge \gamma n'/2$. Hence \cref{lem:hamilton-robust} gives a Hamilton cycle $C'$ in $D'$.

We now determine how many vertices to absorb into each subdivided path. Enumerate $A(F)=\{e_1,\ldots,e_q\}$ and set $s_i:=\ell(e_i)-p_i$ for every $i\in[q]$. For sufficiently large $n$, $s_i\ge(C-C_{\rm A})\log n\ge2$.
Since $F_0$ is an $F$-subdivision, $|V(F_0)|=|V(F)|+\sum_{i=1}^q(p_i-1)$,
while admissibility gives $|V(F)|+\sum_{i=1}^q(\ell(e_i)-1)=n$.
Subtracting these two identities yields $|V(D')|=\sum_{i=1}^q(\ell(e_i)-p_i)=\sum_{i=1}^q s_i$.

Partition the Hamilton cycle $C'$ into consecutive pairwise vertex-disjoint directed paths $Q_1,\ldots,Q_q$ with $|V(Q_i)|=s_i$ for each $i\in[q]$.
By \cref{lem:absorb-paths}, there is an $F$-subdivision $F^*$ with branch map $f$ and $V(F^*)=V(F_0)\cup\bigcup_{i=1}^qV(Q_i)=V(D)$. The subdivided path corresponding to $e_i$ has length $p_i+|V(Q_i)|=p_i+s_i=\ell(e_i)$.
Thus $F^*$ is the required spanning $(F,f,\ell)$-subdivision.
\end{proof}

\section{Constant and mixed path lengths}\label{sec:further}

We now remove the logarithmic lower bound on the prescribed path lengths. The main ingredient is an exact-length path lemma for paths valid for lengths up to a small linear fraction of the order. We use it to embed the shorter prescribed paths first and then apply \cref{thm1} to the remaining paths. This yields the mixed-length result \cref{thm:mixed-length}, from which
\cref{thm3} follows directly.

\subsection{Exact length paths}\label{sec:fixed-constant}

We first find paths of every prescribed length in a suitable range between two prescribed vertices, while avoiding a small forbidden set.

\begin{lemma}
\label{lem:exact-short-paths}
Let $0<\eta\ll \beta\ll \nu\le\tau\ll\gamma<1$. There exist constants $L,n_0\in\mathbb N$ such that the following holds. Let $D$ be an $n$-vertex robust $(\nu,\tau)$-outexpander with $\dzero(D)\ge\gamma n$, where $n\ge n_0$. Let $W\subseteq V(D)$ satisfy $|W|\le\beta n$, and let $x,y\in V(D)\setminus W$ be distinct. Then, for every integer $\ell$ with $L\le \ell\le \eta n$, there is a directed path of length exactly $\ell$ from $x$ to $y$ in $D-W$.
\end{lemma}

\begin{proof}
We first construct a sequence of vertex sets, then use the available inneighbours to choose a path of the prescribed length backwards. Set $D_0:=D-W$ and $n':=|V(D_0)|$. By \cref{prop:delete}, after weakening the parameters, $D_0$ is a robust $(\nu_0,\tau_0)$-outexpander with $\dzero(D_0)\ge\gamma_0n'$, where the constants satisfy $0<\nu_0\le\tau_0<\gamma_0/2$ and depend only on $\nu,\tau,\gamma$. Set $\kappa:=\min\{\nu_0,\gamma_0-\tau_0\}$.

Define $X_1:=N^+_{D_0}(x)$. Recursively, whenever $X_i\ne V(D_0)$, let
\[
X_{i+1}:=
\begin{cases}
\rn^+_{\nu_0,D_0}(X_i),
   & \tau_0n'<|X_i|\le(1-\tau_0)n',\\
V(D_0),
   & |X_i|>(1-\tau_0)n'.
\end{cases}
\]
Once $X_i=V(D_0)$, set $X_j=V(D_0)$ for every $j>i$. Since $|X_1|\ge\gamma_0n'>\tau_0n'$, each application of robust outexpansion increases the size of the current set by at least $\nu_0n'$. Hence, for $r:=\lceil\nu_0^{-1}\rceil+2$, we have $X_i=V(D_0)$ for every $i\ge r$. Set $L:=r$.

Fix an integer $\ell$ with $L\le\ell\le\eta n$. Since $X_\ell=V(D_0)$, we may set $v_\ell:=y$. We choose
$v_{\ell-1},\ldots,v_1$ successively backwards. Suppose that $v_j\in X_j$ has already been chosen for some $2\le j\le\ell$. If $X_j=\rn^+_{\nu_0,D_0}(X_{j-1})$, then $v_j$ has at least $\nu_0n'$ inneighbours in $X_{j-1}$. Otherwise, either $X_{j-1}=V(D_0)$ or $|X_{j-1}|>(1-\tau_0)n'$, and therefore
\[
|N^-_{D_0}(v_j)\cap X_{j-1}|
\ge d^-_{D_0}(v_j)-|V(D_0)\setminus X_{j-1}|
\ge(\gamma_0-\tau_0)n'.
\]
Thus, in every case, $|N^-_{D_0}(v_j)\cap X_{j-1}|\ge\kappa n'$. Since $|W|\le\beta n$, we have $n'\ge(1-\beta)n$. As $\eta\ll\kappa$, for sufficiently large $n$, $\ell+1\le\eta n+1<\kappa n'$. Hence we may choose
$v_{j-1}\in N^-_{D_0}(v_j)\cap X_{j-1}$ distinct from $x$ and all vertices already chosen. Continuing in this
way gives distinct vertices $v_1,\ldots,v_\ell$, all different from $x$, with $v_1\in X_1=N^+_{D_0}(x)$ and
$v_{j-1}v_j\in A(D_0)$ for every $2\le j\le\ell$. Consequently, $xv_1\cdots v_{\ell-1}y$ is a directed path of length exactly $\ell$ from $x$ to $y$ in $D-W$.
\end{proof}

\subsection{Mixed path lengths}\label{subsec:short-budgets}

The next result allows more arcs when the total number of internal vertices prescribed for short paths is small. We first embed these short paths with their exact lengths. After deleting their internal vertices, admissibility bounds the number of remaining arcs, allowing us to apply \cref{thm1} to the remaining long paths.

\begin{theorem}\label{thm:mixed-length}
Let $0<\nu\le\tau\ll\gamma<1$. There exist constants $\mu,\beta>0$, $T\ge1$ and $L,n_0\in\mathbb N$ such that the following holds. Let $D$ be an $n$-vertex robust $(\nu,\tau)$-outexpander with $\dzero(D)\ge\gamma n$, where $n\ge n_0$. Let $F$ be a digraph with $|V(F)|\le\mu n$, let $f:V(F)\to V(D)$ be an injection, and let
$\ell:A(F)\to\mathbb N$ be admissible with $\ell(e)\ge L$ for every $e\in A(F)$. If $\sum_{\substack{e\in A(F)\\ \ell(e)-1<T\log n}} (\ell(e)-1)\le\beta n$, then $D$ contains a spanning $(F,f,\ell)$-subdivision.
\end{theorem}

\begin{proof}
Choose constants $0<\eta\ll b\ll\nu\le\tau\ll\gamma$, and let $L$ be the constant given by \cref{lem:exact-short-paths} with forbidden set parameter $b$. Set $\nu_*:=\nu/2$, $\tau_*:=2\tau$ and $\gamma_*:=\gamma/2$, and choose constants $\alpha_*,\mu_*,C_*$ so that \cref{thm1} applies with these parameters. Choose
\[
 0<\beta\le\min\{b/2,\nu/4,\gamma/2,1/4\},\qquad
 0<\mu\le\min\{b/2,\mu_*/2,1/4\},
\]
and take $T\ge\max\{1,C_*,2/\alpha_*\}$. Choose $n_0$ sufficiently large for all subsequent applications of
\cref{lem:exact-short-paths} and \cref{thm1}, and so that $T\log n+1\le\eta n$ and $n/2$ exceeds the order threshold in \cref{thm1} whenever $n\ge n_0$.

Write $B:=f(V(F))$, $h:=|B|$, and
\[
 S:=\{e\in A(F):\ell(e)-1<T\log n\},\qquad
 r:=\sum_{e\in S}(\ell(e)-1)\le\beta n.
\]
Construct the paths corresponding to $S$ successively. When treating $e=ab\in S$, let $I$ be the set of internal vertices used so far, and take $W:=(B\setminus\{f(a),f(b)\})\cup I$ as the forbidden set. Then
$|W|\le h+r\le bn$ and $f(a),f(b)\notin W$. Since $L\le\ell(e)<T\log n+1\le\eta n$, \cref{lem:exact-short-paths} gives an $f(a)$-to-$f(b)$ path $P_e$ of length $\ell(e)$ in $D-W$, whose internal vertices avoid $B$ and $I$.

Let $R$ be the union of the internal vertices of these paths, and set $D':=D-R$ and $n':=n-r$. Then $|R|=r\le\beta n$, $R\cap B=\varnothing$ and $n'\ge n/2$. By \cref{prop:delete}, $D'$ is
a robust $(\nu_*,\tau_*)$-outexpander, and $\dzero(D')\ge\gamma n-r\ge\gamma_*n'$.

Let $F'$ have vertex set $V(F)$ and arc set $A(F)\setminus S$, retaining all vertices that become isolated. Admissibility gives $h+\sum_{e\in A(F')}(\ell(e)-1)=n'$. Moreover, $A(F')\ne\varnothing$, since otherwise
$n=h+r\le(\mu+\beta)n\le n/2$, a contradiction. Each arc of $F'$ satisfies $\ell(e)-1\ge T\log n$, and hence
\[
 |A(F')|
 \le\frac{n}{T\log n}
 \le\frac{\alpha_*n}{2\log n}
 \le\frac{\alpha_*n'}{\log n'},
 \qquad
 |V(F')|=h\le\frac{\mu_*n}{2}\le\mu_*n'.
\]
Here we used $n'\ge n/2$ and $\log n'\le\log n$. Also,
$\ell(e)\ge T\log n+1\ge C_*\log n'$ for every $e\in A(F')$.
Therefore \cref{thm1} applies to $D'$, $F'$, the same branch map $f$
and the restriction of $\ell$.

The resulting subdivision spans $D'$, and its internal vertices avoid
$B$. The paths $P_e$, $e\in S$, have pairwise disjoint sets of
internal vertices, all contained in the deleted set $R$. Thus adding
these paths preserves the branch map and all prescribed lengths and
gives the required spanning $(F,f,\ell)$-subdivision in $D$.
\end{proof}

For example, let $F$ be a matching of $k+1$ arcs, where
$k=\lfloor cn\rfloor$ and $c>0$ is sufficiently small. Prescribe
length $L$ for the first $k$ arcs and length
$n-k(L+1)-1$ for the last arc. These lengths are admissible,
$|V(F)|=2(k+1)\le\mu n$, and the short paths use
$k(L-1)\le\beta n$ internal vertices for all sufficiently large $n$.
Thus \cref{thm:mixed-length} allows $|A(F)|=\Theta(n)$ in this case.

\subsection{Proof of Theorem~\ref{thm3}}

We now deduce \cref{thm3} directly from \cref{thm:mixed-length}.

\begin{proof}[\textbf{Proof of \cref{thm3}}]
Let $\mu,\beta,T,L$ and $n_0$ be the constants supplied by \cref{thm:mixed-length}. Choose
$0<\alpha\le \beta/T$ and $C\ge L$. Let $D,F,f,\ell$ satisfy the assumptions of \cref{thm3}, and set
$S:=\{e\in A(F):\ell(e)-1<T\log n\}$. Then
\[
 \sum_{e\in S}(\ell(e)-1)
 < T\log n\,|A(F)|
 \le T\log n\cdot\frac{\alpha n}{\log n}
 =\alpha Tn
 \le\beta n.
\]
Since
$|V(F)|\le\mu n$ and $\ell(e)\ge C\ge L$ for every $e\in A(F)$,
all the hypotheses of \cref{thm:mixed-length} are satisfied.
Therefore $D$ contains a spanning $(F,f,\ell)$-subdivision.
\end{proof}

\subsection{The Nash--Williams consequence}

We finally prove the Nash--Williams application stated in
\cref{thm2}.

\begin{proof}[\textbf{Proof of \cref{thm2}}]
Apply \cref{lem:NW-robust} with parameter $\gamma$ to obtain constants
$\gamma_0,\nu,\tau>0$. Then choose $\alpha,\mu,C$ and $n_0$ so that
\cref{thm3} applies with semi-degree parameter $\gamma_0$ and expansion
parameters $\nu,\tau$. The result now follows directly from
\cref{lem:NW-robust} and \cref{thm3}.
\end{proof}

\section{Ore-type conditions}\label{sec:ore}
Recall that $\sigmapm(D):=\min\{d^+(x)+d^-(y):x\ne y,\ xy\notin A(D)\}$.
We call an ordered pair $(x,y)$ a \emph{non-arc} if $x\ne y$ and $xy\notin A(D)$.

\subsection{An Ore-type obstruction}\label{subsec:ore-obstruction}
The following construction shows that an additive term depending on the number of prescribed paths is unavoidable.

\begin{proposition}\label{prop:sharp-Ore}
For every $q\ge2$ and $m\ge3q-2$, there exist an oriented graph $D$
on $n=4m+3q-4$ vertices with $\sigmapm(D)\ge3m+3q-4=3n/4+(3q-4)/4$
and an injection $f:V(F_q)\to V(D)$, where $F_q$ is the
disjoint union of $q$ arcs, such that no spanning $F_q$-subdivision
in $D$ has branch map $f$. In particular, when $q=2$, the condition
$\sigmapm(D)\ge\lceil 3n/4\rceil$ does not guarantee a spanning
linkage joining two prescribed disjoint ordered pairs for infinitely
many $n$.
\end{proposition}

The construction for $q=2$ is illustrated in \cref{fig:ore-extremal}.
\begin{figure}[H]
\centering

\begin{tikzpicture}[
    x=1cm,
    y=1cm,
    line cap=round,
    line join=round,
    part/.style={
        circle,
        draw=black,
        line width=1.05pt,
        minimum size=1.50cm,
        inner sep=0pt
    },
    largepart/.style={
        circle,
        draw=black,
        line width=1.05pt,
        minimum size=2.10cm,
        inner sep=0pt
    },
    crossarrow/.style={
        -{Stealth[length=2.8mm,width=2mm]},
        line width=1.05pt
    }
]

\node[part]      (A) at (0,3.05) {};
\node[part]      (B) at (-2.65,0.55) {};
\node[largepart] (C) at (2.65,0.55) {};

\node[font=\large,anchor=south]
    at ([yshift=0.10cm]A.north) {$A$};

\node[font=\large,anchor=south]
    at ([yshift=0.10cm]B.north) {$B$};

\node[font=\large,anchor=south]
    at ([yshift=0.10cm]C.north) {$C$};

\manualstealtharc{B}{0.43}{215}{505}{-15}
\manualstealtharc{C}{0.62}{215}{505}{-9}

\draw[crossarrow]
    (A.225) -- (B.45);

\draw[crossarrow]
    (B.east) -- (C.west);

\draw[crossarrow]
    (C.135) -- (A.315);

\end{tikzpicture}

\caption{The cyclic obstruction for $q=2$.
The set $A$ is independent with $|A|=m$, while $B$ and $C$ are
tournaments with $|B|=m+1$ and $|C|=2m+1$. All arcs between distinct
parts are oriented cyclically as $C\to A\to B\to C$. The prescribed
paths are directed from $b_i\in B$ to $c_i\in C$.}
\label{fig:ore-extremal}

\end{figure}
\begin{proof}[\textbf{Proof of \cref{prop:sharp-Ore}}]
Let $V(D)=A\cup B\cup C$, where $|A|=m$, $|B|=m+q-1$, and $|C|=2m+2q-3$.
Include all possible arcs between distinct parts, oriented cyclically as $C\to A,A\to B,B\to C$. Let $A$ be independent, orient $B$ as an almost regular tournament, and orient $C$ as a regular tournament. Every vertex of $B$ then has both indegree and outdegree at least $\left\lfloor \frac{m+q-2}{2}\right\rfloor$ inside $B$, and every vertex of $C$ has indegree and outdegree exactly $m+q-2$ inside $C$.

We check that $\sigmapm(D)\ge3m+3q-4$ by considering the possible locations of a non-arc. If $x,y\in A$, then
$d^+(x)+d^-(y)=|B|+|C|=3m+3q-4$. If $x\in A$ and $y\in C$, then
$d^+(x)+d^-(y)=|B|+\bigl(|B|+m+q-2\bigr)=3m+3q-4$. If $x\in B$ and $y\in A$, then $d^+(x)+d^-(y)\ge |C|+|C|=4m+4q-6\ge3m+3q-4$.
If $x\in C$ and $y\in B$, then $d^+(x)+d^-(y)\ge (|A|+m+q-2)+\left(|A|+\left\lfloor\frac{m+q-2}{2}\right\rfloor\right)$. Because $m\ge3q-2$, we have
$\left\lfloor\frac{m+q-2}{2}\right\rfloor\ge2q-2$, and therefore $d^+(x)+d^-(y)\ge3m+3q-4$.
For ordered non-arcs inside $B$, we have $d^+(x)+d^-(y)\ge |C|+|A|+2\left\lfloor\frac{m+q-2}{2}\right\rfloor\ge3m+3q-4$. For ordered non-arcs inside $C$, we have $d^+(x)+d^-(y)=(|A|+m+q-2)+(|B|+m+q-2)=4m+3q-5\ge3m+3q-4$. These cases exhaust all non-arcs, so $\sigmapm(D)\ge3m+3q-4$.

Now let $F_q$ be the disjoint union of $q$ arcs $u_1v_1,\ldots,u_qv_q$. Choose distinct vertices $b_1,\ldots,b_q\in B$ and $c_1,\ldots,c_q\in C$, and define
$f(u_i)=b_i,\qquad f(v_i)=c_i$ for $i=1,\ldots,q$. Suppose that a spanning $F_q$-subdivision with branch map $f$ exists. It consists of $q$ vertex-disjoint directed paths from $b_i$ to $c_i$, $i=1,\dots,q$, which together cover $V(D)$.

Every vertex of $A$ has all its outneighbours in $B$ and cannot be a terminal vertex of a prescribed path, since all terminal vertices lie in $C$. Thus the $m$ vertices of $A$ have distinct immediate successors in $B$. None of these successors can be an initial vertex $b_i$, so all $m$ successors must lie in
$B\setminus\{b_1,\ldots,b_q\}$, which has size $|B|-q=m-1$. This is a contradiction. Hence no spanning $F_q$-subdivision with branch map $f$ exists. For $q=2$, we have $n=4m+2$ and $3m+2=\lceil3n/4\rceil$.
\end{proof}

The obstruction is independent of the prescribed path lengths. It shows that, for $F_q$, an additive term of order $q$ is unavoidable in the degree condition.

\begin{remark}\label{rem:Ore-isolated}
If $F$ consists of $q\ge1$ disjoint arcs and $r$ isolated vertices, where $q+r\ge2$, use the construction above with $q$ replaced by $q+r$. Map the $q$ initial vertices and the $r$ isolated vertices to distinct vertices of $B$, and the $q$ final vertices to distinct vertices of $C$. The same successor count rules out a spanning $F$-subdivision with this branch map, while $\sigmapm(D)\ge3n/4+(3(q+r)-4)/4$. In particular, for every fixed $K>0$, taking $q=1$ and $3r-1>4K$ shows that $\sigmapm(D)\ge3n/4+Kq$ cannot suffice when isolated vertices are allowed. With $q,r$ fixed, $m$ can be arbitrarily large, so the size bounds and lower bounds on path lengths do not remove this obstruction.
\end{remark}

\subsection{Robust expansion from an Ore-type condition}

To apply \cref{thm3}, we show that the Ore-type condition implies
both robust outexpansion and a linear minimum semi-degree bound.

\begin{proposition}\label{prop:Ore-expansion}
For every $\eps>0$ and $\zeta>0$ there exist constants $0<\nu\le\tau\le\zeta$ and $n_0\in\mathbb N$ such that every oriented graph $D$ on $n\ge n_0$ vertices with $\sigmapm(D)\ge (3/4+\eps)n$ is a robust $(\nu,\tau)$-outexpander.
\end{proposition}

We also need the following semi-degree estimate.

\begin{lemma}
\label{lem:Ore-semi-degree}
Let $D$ be an oriented graph on $n\ge2$ vertices. Then
\[
 \dzero(D)\ge\frac{2\sigmapm(D)-n+2}{3}.
\]
In particular, if $\sigmapm(D)\ge(3/4+\eps)n$, then $\dzero(D)\ge(1/6+2\eps/3)n+2/3$.
\end{lemma}
\begin{proof}
Set $c:=\sigmapm(D)/n$. Let $x$ be a vertex of minimum outdegree, and set $a:=d^+(x)/n$ and $Y:=V(D)\setminus(N^+(x)\cup\{x\})$, so $|Y|=(1-a)n-1$. Since the average outdegree is at most $(n-1)/2$, the set $Y$ is nonempty. For every $y\in Y$, the pair $(x,y)$ is a non-arc, so $d^-(y)\ge(c-a)n$. There are no arcs from $x$ to $Y$, and an oriented graph has at most $\binom{|Y|}{2}$ arcs inside $Y$. Therefore
\[
 |Y|(c-a)n\le\sum_{y\in Y}d^-(y)
 \le an|Y|+\binom{|Y|}{2}
 =|Y|\frac{(1+a)n-2}{2}.
\]
Dividing by $|Y|$ gives $3an\ge(2c-1)n+2$, and hence $\delta^+(D)\ge(2\sigmapm(D)-n+2)/3$. The indegree argument is symmetric, so the same lower bound holds for $\delta^-(D)$.
\end{proof}

\begin{proof}[\textbf{Proof of \cref{prop:Ore-expansion}}]
We may assume $0<\eps<1/4$. Let $c:=3/4+\eps$ and choose $0<1/n_0\ll\nu\ll\rho\ll\tau\ll\min\{\eps,\zeta\}$,
with $\rho\le\tau/4$ and $\nu/\rho^2\ll\eps$. Suppose that $D$ is not a robust $(\nu,\tau)$-outexpander. Then there exists a set $S\subseteq V(D)$ with
$\tau n\le|S|\le(1-\tau)n$ such that $|T|<|S|+\nu n$, where $T:=\rn^+_{\nu,D}(S)$. Set $B:=V(D)\setminus T$. Every vertex of $B$ has fewer than $\nu n$ inneighbours in $S$, so
\begin{equation}\label{eq:SB-small}
 e_D(S,B)<\nu n|B|\le\nu n^2.
\end{equation}
Moreover, $|B|=n-|T|>n-|S|-\nu n\ge(\tau-\nu)n$, so $|B|\ge\tau n/2$.

Partition $V(D)$ into $X:=S\setminus B$, $Y:=S\cap B$, $Z:=B\setminus S$, and $W:=V(D)\setminus(S\cup B)$, and write $x:=|X|/n$, $y:=|Y|/n$, $z:=|Z|/n$, and $w:=|W|/n$. Since $W=T\setminus S$ and $Y=S\setminus T$, the inequality $|T|<|S|+\nu n$ gives
\begin{equation}\label{eq:w-y}
 w<y+\nu.
\end{equation}
For a nonempty set $U$, write $\overline d^+(U):=|U|^{-1}\sum_{u\in U}d^+(u)$ and $\overline d^-(U):=|U|^{-1}\sum_{u\in U}d^-(u)$.

The Ore-type condition gives the following lower bound on average degree sums, which we will contradict by counting arcs between the four parts.

\begin{claim}\label{claim:averaging}
If $U\subseteq S$ and $V\subseteq B$ satisfy $|U|,|V|\ge\rho n$, then
$\overline d^+(U)+\overline d^-(V)>\left(\frac34+\frac\eps2\right)n$.
\end{claim}
\begin{proof}\renewcommand*{\qedsymbol}{$\blacksquare$}
By \eqref{eq:SB-small}, at most $\nu n^2$ pairs in $U\times V$ correspond to arcs of $D$. There are also at most $|U\cap V|\le n$ diagonal pairs $(v,v)$. Hence at least $|U||V|-\nu n^2-n$ pairs in $U\times V$ are ordered non-arcs. For sufficiently large $n$, this is at least $(1-2\nu/\rho^2)|U||V|$. Applying the Ore-degree condition to these non-arcs and using the nonnegativity of all other degree sums, we obtain $\overline d^+(U)+\overline d^-(V)\ge(1-2\nu/\rho^2)cn >(3/4+\eps/2)n$.
\end{proof}

First suppose that $y\ge1/4-\eps/16$. Then $|Y|\ge\rho n$. Arcs between $Y$ and its complement contribute at most $|Y|(n-|Y|)$ to $\sum_{v\in Y}(d^+(v)+d^-(v))$, while arcs inside $Y$ are counted twice and are included in $e_D(S,B)$. Thus \eqref{eq:SB-small} gives $\overline d^+(Y)+\overline d^-(Y)\le(1-y+2\nu/y)n\le(3/4+\eps/2)n$, contradicting \cref{claim:averaging}. Hence
\begin{equation}\label{eq:y-small}
 y<1/4-\eps/16.
\end{equation}

It remains to bound the relevant average degree sums when \eqref{eq:y-small} holds. The following claim covers the three possible cases; its estimates are proved in \cref{app:ore-estimates}.

\begin{claim}\label{claim:ore-cases}
Suppose that \eqref{eq:SB-small}, \eqref{eq:w-y}, and \eqref{eq:y-small} hold. Then:
\begin{enumerate}[label=\textup{(\roman*)}]
 \item If $x,z\ge\rho$, then $\overline d^+(X)+\overline d^-(Z)\le(3/4+\eps/2)n$.
 \item If $x<\rho$, then $y,z\ge\rho$ and $\overline d^+(Y)+\overline d^-(Z)\le(3/4+\eps/2)n$.
 \item If $z<\rho$, then $x,y\ge\rho$ and $\overline d^+(X)+\overline d^-(Y)\le(3/4+\eps/2)n$.
\end{enumerate}
\end{claim}

If $x,z\ge\rho$, part~\textup{(i)} contradicts \cref{claim:averaging}. Otherwise, either $x<\rho$ or $z<\rho$, and part~\textup{(ii)} or part~\textup{(iii)}, respectively, gives the same contradiction. This proves the proposition.
\end{proof}
\begin{proof}[\textbf{Proof of \cref{thm4}}]
By \cref{lem:Ore-semi-degree}, we may take $\eta:=1/6$. Choose $0<\zeta\ll\eta$, and apply \cref{prop:Ore-expansion} with this value of $\zeta$ to obtain constants
$0<\nu\le\tau\le\zeta$. Choose $\alpha,\mu,C$ and $n_0$ as supplied by \cref{thm3} for the parameters $\nu,\tau,\eta$, and increase $n_0$ so that all preceding results apply. Every oriented graph $D$ satisfying $\sigmapm(D)\ge(3/4+\eps)n$ then has $\dzero(D)\ge\eta n$ and is a robust
$(\nu,\tau)$-outexpander. The result now follows from \cref{thm3}.
\end{proof}

The same reduction, using \cref{thm:mixed-length}, gives the corresponding mixed-length conclusion under the Ore-type condition.

\section{Concluding remarks}\label{sec:concluding}

\Cref{thm:mixed-length} permits linearly many target arcs
under an additional constraint on the total number of
internal vertices prescribed for short paths.
A natural question is whether the arc bound in \cref{thm3}
can be relaxed to $|A(F)|\le\alpha n$ for some constant
$\alpha>0$, while retaining the conclusion for every
admissible length assignment whose values are bounded
below by a sufficiently large constant.

The construction in \cref{prop:sharp-Ore} motivates the
following questions about additive Ore-type conditions.

\begin{problem}\label{prob:additive-Ore}
Let $H$ be a fixed digraph.
Do there exist $C_H>0$ and $L_H\in\mathbb N$ such that,
for all sufficiently large $n$, every $n$-vertex oriented
graph $D$ with $\sigmapm(D)\ge 3n/4+C_H$ contains a spanning
$(H,f,\ell)$-subdivision for every injection
$f:V(H)\to V(D)$ and every admissible
$\ell:A(H)\to\mathbb N$ satisfying $\ell(e)\ge L_H$
for all $e\in A(H)$?

For target digraphs that may grow with $n$, do there exist
$K,\alpha>0$ and $L\in\mathbb N$ such that,
for all sufficiently large $n$ and every digraph $F$
without isolated vertices and with
$q:=|A(F)|\le\alpha n/\log n$, every $n$-vertex oriented
graph $D$ with $\sigmapm(D)\ge 3n/4+Kq$ contains a spanning
$(F,f,\ell)$-subdivision for every injection
$f:V(F)\to V(D)$ and every admissible
$\ell:A(F)\to\mathbb N$ satisfying $\ell(e)\ge L$
for all $e\in A(F)$?
\end{problem}

The restriction on isolated vertices in the second question
is necessary for an additive correction of the form $Kq$,
see \cref{rem:Ore-isolated}.
If isolated vertices are allowed, a natural alternative is
an additive correction of the form $K(|V(F)|+|A(F)|)$.
\appendix
\section{Proof of Claim~\ref{claim:ore-cases}}\label{app:ore-estimates}

We use the notation for average degrees from the proof of \cref{prop:Ore-expansion}, together with the identity $x+y+z+w=1$.
Suppose first that $x,z\ge\rho$. There are at most $\binom{|X|}{2}$ arcs from $X$ to $X$ and at most $|X||W|$ arcs from $X$ to $W$. All remaining out-arcs from $X$ end in $Y\cup Z\subseteq B$ and are counted in $e_D(S,B)$. Hence
\begin{equation}\label{eq:X-out}
 \overline d^+(X)\le\left(\frac{x}{2}+w+\frac{\nu}{x}\right)n.
\end{equation}
Similarly, there are at most $\binom{|Z|}{2}$ arcs from $Z$ to $Z$ and at most $|Z||W|$ arcs from $W$ to $Z$. All remaining in-arcs to $Z$ originate in $X\cup Y\subseteq S$ and are counted in $e_D(S,B)$. Thus
\begin{equation}\label{eq:Z-in}
 \overline d^-(Z)\le\left(\frac{z}{2}+w+\frac{\nu}{z}\right)n.
\end{equation}
Using $x+z=1-y-w$, \eqref{eq:w-y}, and \eqref{eq:y-small}, we obtain
\begin{align*}
\overline d^+(X)+\overline d^-(Z)
\le\left(\frac12-\frac y2+\frac{3w}{2}+\frac{2\nu}{\rho}\right)n
\le\left(\frac12+y+\frac{3\nu}{2}+\frac{2\nu}{\rho}\right)n
 <\left(\frac34+\frac\eps2\right)n.
\end{align*}
This proves~\textup{(i)}.

Now suppose that $x<\rho$. Since $|S|\ge\tau n$ and $S=X\cup Y$, we have $y\ge\tau-\rho\ge\rho$. Also $|S|+|B|=|X|+2|Y|+|Z|$ and $|B|>n-|S|-\nu n$, so $x+2y+z>1-\nu$. Together with $x<\rho$ and \eqref{eq:y-small}, this yields $z>1/2+\eps/8-\rho-\nu\ge\rho$. Out-arcs from $Y$ to $X\cup W$ contribute at most $|Y|(|X|+|W|)$, while all remaining out-arcs from $Y$ end in $Y\cup Z\subseteq B$. Therefore $\overline d^+(Y)\le(x+w+\nu/y)n$. Combining this with \eqref{eq:Z-in} and using $z=1-x-y-w$, \eqref{eq:w-y}, and \eqref{eq:y-small}, we get
\begin{align*}
\overline d^+(Y)+\overline d^-(Z)
\le\left(\frac12+\frac x2-\frac y2+\frac{3w}{2}+\frac{2\nu}{\rho}\right)n
\le\left(\frac12+\frac\rho2+y+\frac{3\nu}{2}+\frac{2\nu}{\rho}\right)n
 <\left(\frac34+\frac\eps2\right)n.
\end{align*}
This proves~\textup{(ii)}.

Finally suppose that $z<\rho$. Since $|B|\ge\tau n/2$ and $B=Y\cup Z$, we have $y\ge\tau/2-\rho\ge\rho$. The inequality $x+2y+z>1-\nu$, together with $z<\rho$ and \eqref{eq:y-small}, gives $x>1/2+\eps/8-\rho-\nu\ge\rho$. In-arcs to $Y$ from $Z\cup W$ contribute at most $|Y|(|Z|+|W|)$, while all remaining in-arcs to $Y$ originate in $X\cup Y\subseteq S$. Hence $\overline d^-(Y)\le(z+w+\nu/y)n$. Combining this with \eqref{eq:X-out} and using $x=1-y-z-w$, \eqref{eq:w-y}, and \eqref{eq:y-small}, we obtain
\begin{align*}
\overline d^+(X)+\overline d^-(Y)
\le\left(\frac12-\frac y2+\frac z2+\frac{3w}{2}+\frac{2\nu}{\rho}\right)n
\le\left(\frac12+y+\frac\rho2+\frac{3\nu}{2}+\frac{2\nu}{\rho}\right)n
 <\left(\frac34+\frac\eps2\right)n.
\end{align*}
This proves~\textup{(iii)}.

\end{document}